\documentclass[letterpaper,12pt]{article}

\title{Parametric Integer Programming in Fixed Dimension}

\author{Friedrich Eisenbrand and Gennady Shmonin \medskip \\ {\small
  Institut f\"ur Mathematik, Universit\"at Paderborn, D-33095
  Paderborn, Germany}}
\date{}

\usepackage{amssymb,amsbsy,latexsym}
\usepackage{mathptmx}
\usepackage[scaled=0.92]{helvet}
\usepackage{courier}
\usepackage[latin1]{inputenc}
\usepackage[T1]{fontenc}

\usepackage[german,dutch,english]{babel}

\usepackage{amsmath,amscd,amsthm}
\usepackage[round,authoryear]{natbib}
\usepackage[dvips]{graphicx,epsfig,color}
\usepackage{pstricks}
\usepackage{pst-node} 
\newpsobject{showgrid}{psgrid}{subgriddiv=1,griddots=10,gridlabels=6pt}

\usepackage{paralist}

\usepackage[dvips,letterpaper,pdfstartview=FitH,unicode]{hyperref}
\usepackage{url}
\newcommand\pdfhypersetup{
  \hypersetup{ bookmarksopen=false, colorlinks=false,
  pdfpagemode=UseOutlines, pdfstartview=FitH, pdftitle=\pdftitle,
  pdfauthor=\pdfauthor, pdfsubject=\pdfsubject,
  pdfkeywords=\pdfkeywords } }

\theoremstyle{plain}
\newtheorem{theorem}{Theorem}[section]
\newtheorem{lemma}[theorem]{Lemma}

\theoremstyle{definition}

\theoremstyle{remark}
\newtheorem*{remark}{Remark}

\providecommand{\abs}[1]{\lvert#1\rvert}

\providecommand{\coloneqq}{\mathrel{\mathop{:}}=}

\providecommand{\setZ}{\mathbb{Z}}
\providecommand{\setQ}{\mathbb{Q}}
\providecommand{\setR}{\mathbb{R}}

\renewcommand{\leq}{\leqslant}
\renewcommand{\geq}{\geqslant}

\newcommand{\size}{\mathrm{size}}

\def\pdftitle{Parametric Integer Programming in Fixed Dimension}
\def\pdfauthor{Friedrich Eisenbrand, Gennady Shmonin}
\def\pdfsubject{Paper draft}
\def\pdfkeywords{parametric integer programming; fixed dimension;
  integer programming gap}
\pdfhypersetup

\begin{document}
\newgray{vlg}{.75}
\newgray{vvlg}{.85}

\maketitle

 \begin{abstract}
   \noindent We consider the following problem: Given a rational
   matrix $A \in \setQ^{m \times n}$ and a rational polyhedron $Q
   \subseteq\setR^{m+p}$, decide if for all vectors $b \in \setR^m$,
   for which there exists an integral $z \in \setZ^p$ such that $(b,
   z) \in Q$, the system of linear inequalities $A x \leq b$ has an
   integral solution. We show that there exists an algorithm that
   solves this problem in polynomial time if $p$ and $n$ are
   fixed. This extends a result of Kannan (1990) who established such
   an algorithm for the case when, in addition to $p$ and $n$, the
   affine dimension of $Q$ is fixed. \medskip

   \noindent As an application of this result, we describe an
   algorithm to find the maximum difference between the optimum values
   of an integer program $\max \{ c x : A x \leq b, \, x \in \setZ^n
   \}$ and its linear programming relaxation over all right-hand sides
   $b$, for which the integer program is feasible. The algorithm is
   polynomial if $n$ is fixed. This is an extension of a recent result
   of Ho\c{s}ten and Sturmfels (2003) who presented such an algorithm
   for integer programs in standard form.
 \end{abstract}

\section{Introduction}
\label{sec:intro}

 Central to this paper is the following \emph{parametric integer
 linear programming} (\emph{PILP}) problem:
 \begin{quote}
   Given a rational matrix $A \in \setQ^{m \times n}$ and a rational
   polyhedron $Q \subseteq \setR^{m+p}$, decide if for all $b \in
   \setR^m$, for which there exists an integral $z \in \setZ^p$ such
   that $(b, z) \in Q$, the system of linear inequalities $A x \leq b$
   has an integral solution.
 \end{quote}
 In other words, we need to check that for all vectors $b$ in the set
 \[
   Q / \setZ^p \coloneqq \{ b \in \setQ^m : (b, z) \in Q \text{ for
   some } z \in \setZ^p \}
 \]
 the corresponding \emph{integer linear programming} problem $A x \leq
 b$, $x \in \setZ^n$ has a feasible solution. The set $Q / \setZ^p$ is
 called the \emph{integer projection} of $Q$. Using this notation, we
 can reformulate PILP as the problem of testing the following
 $\forall\,\exists$-sentence:
 \begin{equation}
   \label{eq:pilp}
   \forall b \in Q / \setZ^p \quad \exists x \in \setZ^n : \quad A x
   \leq b.
 \end{equation}
 It is worth noticing that any polyhedron $Q \subseteq \setR^m$ as
 well as the set of integral vectors in $Q$ can be expressed by means
 of integer projections of polyhedra. Indeed,
 \[
   Q = Q / \setZ^0 \quad \text{and} \quad Q \cap \setZ^m = \{ (b, b) :
   b \in Q \} / \setZ^m.
 \]

 In its general form, PILP belongs to the second level of the
 polynomial hierarchy and is $\Pi_2^p$-complete; see
 \citep{Stockmeyer1976} and \citep{Wrathall1976}. \cite{Kannan1990}
 presented a polynomial algorithm to decide the sentence
 \eqref{eq:pilp} in the case when $n$, $p$ and the affine dimension of
 $Q$ are fixed. This result was applied to deduce a polynomial
 algorithm that solves the Frobenius problem when the number of input
 integers is fixed, see \citep{Kannan1992}.

 Kannan's algorithm proceeds in several steps. We informally describe
 it at this point as a way to decide $\forall \, \exists$-statements
 \eqref{eq:pilp} in the case $p=0$. First Kannan provides an algorithm
 which partitions the set of right-hand sides $Q$ into polynomially
 many \emph{integer projections of partially open polyhedra} $S_1,
 \ldots, S_t$, where each $S_i$ is obtained from a higher-dimensional
 polyhedron by projecting out a fixed number of integer variables.
 Each $S_i$ is further equipped with a fixed number of mixed integer
 programs such that for each $b \in S_i$ the system $A x \leq b$ is
 integer feasible, if and only if one of the fixed number of
 ``candidate solutions'' obtained from plugging $b$ in these
 associated mixed integer programs, is a feasible integer point.

 To decide now whether \eqref{eq:pilp} holds, one searches within the
 sets $S_i$ individually for a vector $b$ for which $A x \leq b$ has
 \emph{no} integral solution. In other words, each of the candidate
 solutions associated to $b$ must violate at least one of the
 inequalities in $A x \leq b$. Since the number of candidate solutions
 is fixed, we can enumerate the choices to associate a violated
 inequality to each candidate solution. Each of these polynomially
 many choices yields now a mixed-integer program with a fixed number
 of integer variables. There exists a $b \in S_i$ such that $A x \leq
 b$ has no integral solution if and only if one of these mixed-integer
 programs is feasible. The latter can be checked with the algorithm of
 \cite{Lenstra1983} in polynomial time.

\subsection*{Contributions of this paper} 

 We modify the algorithm of Kannan to run in polynomial time under the
 assumption that only $n$ and $p$ are fixed. This is achieved via
 providing an algorithm that computes for a matrix $A \in \setQ^{m
 \times n}$ a set $D \subseteq \setZ^n$ of integral directions with
 the following property: for each $b \in \setR^m$, the lattice width
 (see Section \ref{sec:flatn-theor-integ}) of the polyhedron $P_b = \{
 x : A x \leq b \}$ is equal to the width of this polyhedron along one
 of the directions in $D$. This algorithm is described in Section
 \ref{sec:width} and runs in polynomial time if $n$ is fixed. The
 strengthening of Kannan's algorithm to decide $\forall \,
 \exists$-statements of the form \eqref{eq:pilp} if $n$ and $p$ is
 fixed follows then by using this result in the proof of Theorem 4.1
 in \citep{Kannan1992}.

 We then apply this result to find the maximum \emph{integer
 programming gap} for a family of integer programs. The integer
 programming gap of an integer program
 \begin{equation}
   \label{eq:ilp-intro}
   \max \{ c x : A x \leq b, \, x \in \setZ^n \}
 \end{equation}
 is the difference
 \[
   \max \{ c x : A x \leq b \} - \max \{ c x : A x \leq b, \, x \in
   \setZ^n \}.
 \]
 Given a rational matrix $A \in \setQ^{m \times n}$ and a rational
 objective vector $c \in \setQ^n$,  $g(A, c)$ denotes  the maximum
 integer programming gap of integer programs of the form
 \eqref{eq:ilp-intro}, where the maximum is taken over all vectors
 $b$, for which the integer program \eqref{eq:ilp-intro} is
 feasible. Our algorithm finds $g(A, c)$ in polynomial time if $n$ is
 fixed. This extends a recent result of \cite{HostenS2003}, who
 proposed an algorithm to find the maximum integer 
 programming gap for a family of integer programs in \emph{standard
 form} if $n$ is fixed. 

\subsection*{Related work}  

 Kannan's algorithm is an extension of the polynomial algorithm for
 integer linear programming in fixed dimension by \cite{Lenstra1983}.
 \cite{BarvinokW2003} presented an algorithm for counting integral
 points in the integer projection $Q / \setZ^p$ of a polytope $Q
 \subseteq \setR^{m+p}$. This algorithm runs in polynomial time if $p$
 and $m$ are fixed, and uses Kannan's partitioning algorithm, which we
 extend in this paper. In particular, their algorithm can be applied
 to count the number of elements of the minimal Hilbert basis of a
 pointed cone in polynomial time if the dimension is fixed. We remark
 that a polynomial test for the Hilbert basis property in fixed
 dimension was first presented by \cite{CookLS1984}. Extensions of
 Barvinok's algorithm to compute counting functions for parametric
 polyhedra were presented in \citep{BarvinokPommersheim99,MR2312001}
 and in \citep{koppe-2007}.  These counting functions are piecewise
 step-polynomials which involve roundup operations.  With these
 functions at hand one can very efficiently compute the number of
 integer points in $P_b$ via evaluation at $b$. It is however not
 known how to use such piecewise step-polynomials to decide
 $\forall\,\exists$-statements efficiently in fixed dimension.

 \cite{HostenS2003} proposed an algorithm to find the maximum integer
 programming gap for a family of integer programs in \emph{standard
 form}, i.e., $\max \{ c x : A x = b, \, x \geq 0, \, x \in \setZ^n
 \}$. Their algorithm exploits short rational generating functions for
 certain lattice point problems, cf. \cite{Barvinok1994} and
 \cite{BarvinokW2003}, and runs in polynomial time if the number $n$
 of columns of $A$ is fixed. However, the latter implies also a fixed
 number of rows in $A$, as we can always assume $A$ to have full row
 rank. We would like to point out that our approach does not rely on
 rational generating functions at all.

\subsection*{Basic definitions and notation}

 For sets $V$ and $W$ in $\setR^n$ and a number $\alpha$ we denote
 \[
   V + W \coloneqq \{ v + w : v \in V, \, w \in W \} \quad \text{and}
   \quad \alpha W \coloneqq \{ \alpha w : w \in W \}.
 \]
 It is easy to see that if $W$ is a convex set containing the origin
 and $\alpha \leq 1$, then $\alpha W \subseteq W$. If $V$ consists of
 one vector $v$ only, we write
 \[
   v + W \coloneqq \{ v + w : w \in W \}
 \]
 and say that $v + W$ is the \emph{translate} of $W$ \emph{along} the
 vector $v$. The symbol $\lceil \alpha \rceil$ denotes the smallest
 integer greater than or equal to $\alpha$, i.e., $\alpha$
 \emph{rounded up}. Similarly, $\lfloor \alpha \rfloor$ stands for the
 largest integer not exceeding $\alpha$, hence $\alpha$ \emph{rounded
 down}.

 In this paper we establish a number of \emph{polynomial algorithms},
 i.e., algorithms whose running time is bounded by a polynomial in the
 input size. Following the standard agreements, we define the
 \emph{size} of a rational number $\alpha = p / q$, where $p, q \in
 \setZ$ are relatively prime and $q > 0$, as the number of bits needed
 to write $\alpha$ in binary encoding:
 \[
   \size(\alpha) \coloneqq 1 + \lceil \log (\abs{p} + 1) \rceil +
   \lceil \log (q + 1) \rceil.
 \]
 The size of a rational vector $a = [ a_1, \ldots, a_n ]$ is the sum
 of the sizes of its components:
 \[
   \size(a) \coloneqq n + \sum_{i=1}^n \size(a_i).
 \]
 At last, the size of a rational matrix $A = [ a_{ij} ] \in \setQ^{m
 \times n}$ is
 \[
   \size(A) \coloneqq m n + \sum_{i=1}^m \sum_{j=1}^n \size(a_{ij}).
 \]

 An \emph{open half-space} in $\setR^n$ is the set of the form $\{ x :
 a x < \beta \}$, where $a \in \setR^n$ is a row-vector and $\beta$ is
 a number. Similarly, the set $\{ x : a x \leq \beta \}$ is called a
 \emph{closed half-space}. A \emph{partially open polyhedron} $P$ is
 the intersection of finitely many closed or open half-spaces. If $P$
 can be defined by means of closed half-spaces only, we say that it is
 a \emph{closed polyhedron}, or simply a \emph{polyhedron}. We need
 the notion of a partially open polyhedron to be able to partition the
 space (this is definitely impossible by means of closed polyhedra
 only). At last, we say that a partially open polyhedron is
 \emph{rational} if it can be defined by the system of linear
 inequalities with rational coefficients and rational right-hand
 sides.

 \emph{Linear programming} is about optimizing a linear function $c x$
 over a given polyhedron $P$ in $\setR^n$:
 \[
   \max \{ c x : x \in P \} = - \min \{ - c x : x \in P \}.
 \]
 If $x$ is required to be integral, it is an \emph{integer linear
 programming} problem
 \[
   \max \{ c x : x \in P \cap \setZ^n \} = - \min \{ - c x : x \in P
   \cap \setZ^n \}.
 \]
 For details on linear and integer programming, we refer to
 \citep{Schrijver1986}. Here we only mention that a linear programming
 problem can be solved in polynomial time, cf. \citep{Khachiyan1979},
 while integer linear programming is NP-complete. However, if the
 number of variables is fixed, integer programming can also be solved
 in polynomial time, as was shown by \cite{Lenstra1983}. Moreover,
 Lenstra presented an algorithm to solve \emph{mixed-integer
 programming} with a fixed number of integer variables. We remark that
 both algorithms---of \cite{Khachiyan1979} and of
 \cite{Lenstra1983}---can be used to solve \emph{decision versions} of
 integer and linear programming on \emph{partially open} polyhedra.

 An integral square matrix $U$ is called \emph{unimodular} if
 $\abs{\mathrm{det}(U)} = 1$. Clearly, if $U$ is unimodular, then
 $U^{-1}$ is also unimodular. A matrix of full row rank is said to be
 in \emph{Hermite normal form} if it has the form $[ H \,\,\, 0 ]$,
 where $H = [ h_{ij} ]$ is a square non-singular non-negative
 upper-triangular matrix such that $h_{ii} > h_{ij}$ for all $j >
 i$. Given a matrix $A$ of full row rank, we can find in polynomial
 time a unimodular matrix $U$ such that $A U$ is in Hermite normal
 form; see \citep{KannanB1979}. We remark that the Hermite normal form
 of an integral vector $c$ is the vector $\alpha e_1$, where $\alpha$
 is the greatest common divisor of the components of $c$ and $e_1$ is
 the first unit vector. The unimodular matrix $U$ such that $c U =
 \alpha e_1$ can be obtained directly while executing the Euclidean
 algorithm to compute the greatest common divisor.

\section{Flatness theorem}
\label{sec:flatn-theor-integ}

 We briefly review the algorithm to solve integer linear programming
 in fixed dimension, as its basic ideas will be used in the following
 sections. Intuitively, if a polyhedron contains no integral point,
 then it must be ``flat'' along some integral direction. In order to
 make this precise, we introduce the notion of ``lattice width.'' The
 \emph{width} $w_c(K)$ of a closed convex set $K$ along a direction $c
 \in \setR^n$ is defined as
 \begin{equation}
   \label{eq:width-c}
   w_c(K) \coloneqq \max \{ c x : x \in K \} - \min \{ c x : x \in K
   \}.
 \end{equation}
 The \emph{lattice width} $w(K)$ of $K$ (with respect to the
 \emph{standard lattice} $\setZ^n$) is the minimum of its widths along
 all non-zero integral directions:
 \[
   w(K) \coloneqq \min \{ w_c(K) : c \in \setZ^n \setminus \{ 0 \} \}.
 \]
 An integral row-vector $c$ attaining the above minimum is called a
 \emph{width direction} of the set $K$. Clearly, $w(v + \alpha K) =
 \alpha w(K)$ for any rational vector $v$ and any non-negative
 rational number $\alpha$. Moreover, both sets $K$ and $v + \alpha K$
 have the same width direction.

 Applications of the concept of lattice width in algorithmic number
 theory and integer programming rely upon the \emph{flatness theorem},
 which goes back to \cite{Khinchin1948} who first proved it for
 ellipsoids in $\setR^n$. Here we state it for \emph{convex bodies},
 i.e., bounded closed convex sets of non-zero volume.

 \begin{theorem}[Flatness theorem]
   \label{thm:flatness}
   There is a constant $\omega(n)$, depending only on $n$, such that
   any convex body $K \subseteq \setR^n$ with $w(K) \geq \omega(n)$
   contains an integral point.
 \end{theorem}

 \noindent The constant $\omega(n)$ in Theorem \ref{thm:flatness} is
 referred to as the \emph{flatness constant}. The best known
 \emph{upper bound} on $\omega(n)$ is $O ( n^{3/2} )$, cf.
 \citep{BanaszczykLPS1999}, although a linear dependence on $n$ was
 conjectured, e.g., by \cite{KannanL1988}. A linear \emph{lower bound}
 on $\omega(n)$ was shown by \cite{Kantor1999} and \cite{Sebo1999}.

 Throughout this paper we will mostly deal with rational polyhedra
 rather than general convex bodies. In this case, assumptions of
 non-zero volume and boundedness can safely be removed from the
 theorem's statement. Indeed, if $P \subseteq \setR^n$ is a rational
 polyhedron of zero volume, then it has width $0$ along an integral
 direction orthogonal to its (rational) affine hull. Further, let $C$
 be the characteristic cone of $P$:
 \[
   C \coloneqq \{ y : x + y \in P \text{ for all } x \in P \}.
 \]
 If $C = \{ 0 \}$, then $P$ is already bounded. If $C$ is
 full-dimensional, then the set $x + C$ trivially contains an integral
 point, for any $x \in P$ (we can always allocate a unit box inside a
 full-dimensional cone). At last, if $C$ is not full dimensional, then
 we can choose a sufficiently large box $B \subseteq \setR^n$ such
 that $w(P) = w(P \cap B)$ and both $P$ and $P \cap B$ have the same
 width direction, which is orthogonal to the (rational) affine hull of
 $C$. If $w(P) \geq \omega(n)$, then $P \cap B$, and hence $P$,
 contains an integral point by Theorem \ref{thm:flatness}.

 How can we use this theorem to check whether a given rational
 polyhedron contains an integral point? The answer is in the following
 lemma, which is almost a direct consequence of the flatness theorem.

 \begin{lemma}
   \label{lem:flatness-app}
   Let $P \subseteq \setR^n$ be a rational polyhedron of finite
   lattice width and let $c$ be its width direction. Let
   \begin{equation}
     \label{eq:beta-intro}
     \beta \coloneqq \min \{ c x : x \in P \}.
   \end{equation}
   Then $P$ contains an integral point if and only if the polyhedron
   \[
     P \cap \{ x : \beta \leq c x \leq \beta + \omega(n) \}
   \]
   contains an integral point.
 \end{lemma}

 \begin{figure}[ht]
   \centering {
     \psset{xunit=.5cm,yunit=.4cm}
     \begin{pspicture}(0,-5)(12,8)       
       \pspolygon[fillstyle=solid,fillcolor=vvlg]
         (0,0)(3,6)(6,7.5)(12,0)(6,-3)(3,-3)
       \psline(0,-4)(0,8)
       \rput[l](0,-4.5){$cx = \beta$}

       \pspolygon[fillstyle=solid,fillcolor=lightgray]
         (0,0)(1,2)(2,2.5)(4,0)(2,-1)(1,-1)
       \psline(4,-4)(4,8)
       \rput[l](4,-4.5){$cx = \beta+\omega(n)$}
       \rput(3,0){$P'$}
       \rput(10,0){$P$}
     \end{pspicture}
   }
   \caption{Illustration for the proof of Lemma
     \ref{lem:flatness-app}}
   \label{fig:1}
 \end{figure}
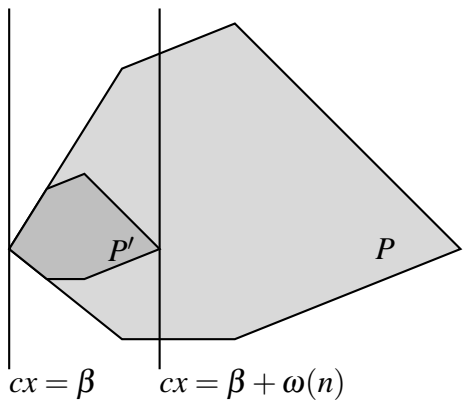

 \begin{proof}
   If $w(P) < \omega(n)$, then there is nothing to prove, since
   \[
     P \subseteq \{ x : \beta \leq c x < \beta + \omega(n) \}.
   \]
   Suppose that $w(P) \geq \omega(n)$ and let $P = y + Q$, where $y$
   is an optimum solution of the linear program \eqref{eq:beta-intro}
   and $Q$ is the polyhedron containing the origin,
   \[
     Q \coloneqq \{ x - y : x \in P \}.
   \]
   We denote
   \[
     Q' \coloneqq \textstyle \frac{\omega(n)}{w(P)} Q \quad \text{and}
     \quad P' \coloneqq y + Q'.
   \]
   In other words, $Q$ is $P$ translated to contain the origin, $Q'$
   is obtained from $Q$ by scaling it down, and $P'$ is $Q'$
   translated back to the original position. It is easy to see that
   \[
     \min \{ c x : x \in P' \} = c y = \beta.
   \]
   Since $\frac{\omega(n)}{w(P)} \leq 1$ and $Q$ is convex, we have
   $Q' \subseteq Q$. This implies $P' \subseteq P$. Yet, we have $w(P)
   = w(Q)$, and therefore, $w(P') = w(Q') = \omega(n)$.

   By Theorem \ref{thm:flatness}, $P'$ contains an integral point, say
   $z$. But then $z$ also belongs to $P$ and
   \[
     c z \leq \max \{ c x : x \in P' \} = \beta + \omega(n).
   \]
   This completes the proof.
 \end{proof}

 \noindent Suppose that we know a width direction $c$ of a polyhedron
 \begin{equation}
   \label{eq:polyhedron-intro}
   P = \{ x : A x \leq b \} \subseteq \setR^n.
 \end{equation}
 Since $c$ is integral, the scalar product $c x$ must be an integer
 for any integral point $x \in P$. Together with Lemma
 \ref{lem:flatness-app}, it allows us to split the original problem
 into $\omega(n) + 1$ integer programming problems on
 lower-dimensional polyhedra
 \[
   P \cap \{ x : c x = \lceil \beta \rceil + j \}, \quad j = 0,
   \ldots, \omega(n)
 \]
 where $\beta$ is defined by \eqref{eq:beta-intro}.

 The components of $c$ must be relatively prime, as otherwise we could
 scale $c$, obtaining a smaller width of $P$. Therefore its Hermite
 normal form is a unit row-vector $e_1$. We can find a unimodular
 matrix $U$ such that $c U = e_1$, introduce new variables $y
 \coloneqq U^{-1} x$ and rewrite the original system of linear
 inequalities $A x \leq b$ in the form $A U y \leq b$. Since $U$ is
 unimodular, the system $A x \leq b$ has an integral solution if and
 only if the system $A U y \leq b$ has an integral solution. But the
 equation $c x = \lceil \beta \rceil + j$ turns into $e_1 y = \lceil
 \beta \rceil + j$. Thus, the first component of $y$ can be
 eliminated. All together, we can proceed with a constant number of
 integer programming problems with a smaller number of variables. If
 $n$ is fixed, this yields a polynomial algorithm.

 An attempt to generalize this approach for the case of varying $b$
 gives rise to the following problems. First, the width directions of
 the polyhedron \eqref{eq:polyhedron-intro} depend on $b$ and
 therefore can also vary. Furthermore, even if a width direction $c$
 remains the same, it is not a trivial task to proceed
 recursively. The point is that $\beta$, as it is defined in
 \eqref{eq:beta-intro}, also depends on $b$ and the hyper-planes $\{ x
 : c x = \lceil \beta \rceil + j \}$ are not easy to construct with
 $\beta$ being a \emph{function} of $b$. In the following sections we
 basically resolve these two problems and adapt the above algorithm
 for the case of varying $b$.

\section{Lattice width of a parametric polyhedron}
\label{sec:width}

 A rational \emph{parametric polyhedron} $P$ defined by a matrix $A
 \in \setQ^{m \times n}$ is the family of polyhedra of the form
 \[
   P_b \coloneqq \{ x : A x \leq b \},
 \]
 where the right-hand side $b$ is allowed to vary over $\setR^m$. We
 restrict our attention only to those $b$, for which $P_b$ is
 non-empty. For each such $b$, there is a width direction $c$ of the
 polyhedron $P_b$. We aim to find a small set $C$ of non-zero integral
 directions such that
 \[
   w(P_b) = \min \{ w_c(P_b) : c \in C \}
 \]
 for all vectors $b$ for which $P_b$ is non-empty. Further on, the
 elements of the set $C$ are referred to as \emph{width directions} of
 the parametric polyhedron $P$. It turns out that such a set can be
 computed in polynomial time when the number of columns in $A$ is
 fixed.

 Let $A \in \setQ^{m \times n}$ be a matrix of full column rank. Given
 a subset of indices
 \[
   N = \{ i_1, \ldots, i_n \} \subseteq \{ 1, \ldots, m \},
 \]
 we denote by $A_N$ the matrix composed of the rows $i_1, \ldots, i_n$
 of $A$. We say that $N$ is a \emph{basis} of $A$ if $A_N$ is
 non-singular. Clearly, any matrix of full column rank has at least
 one basis. Each basis $N$ defines a linear transformation
 \begin{equation}
   \label{eq:fn}
   F_N : \setR^m \to \setR^n, \quad F_N b = A_N^{-1} b_N,
 \end{equation}
 which maps right-hand sides $b$ to the corresponding \emph{basic
 solutions}. We can view $F_N$ as an $n \times m$-matrix of rational
 numbers. If the point $F_N b$ satisfies the system $A x \leq b$, then
 it is a vertex of the polyhedron $\{ x : A x \leq b \}$. From linear
 programming duality we know that the optimum value of any feasible
 linear program
 \[
   \max \{ c x : A x \leq b \}
 \]
 is finite if and only if there is a basis $N$ such that $c = y A_N$
 for some row-vector $y \geq 0$. In other words, $c$ must belong to
 the cone generated by the rows of matrix $A_N$. Moreover, if it is
 finite, there is a basis $N$ such that the optimum value is attained
 at $F_N b$. It gives us the following simple lemma.

 \begin{lemma}
   \label{lem:infinite-width}
   Let $P$ be a parametric polyhedron defined by a rational matrix
   $A$. If there exists a vector $b'$ such that the polyhedron
   \[
     P_{b'} = \{ x : A x \leq b' \}
   \]
   has infinite lattice width, then $w(P_b)$ is infinite for all $b$.
 \end{lemma}

 \begin{proof}
   Suppose that the lattice width of $P_b$ is finite for some $b$ and
   let $c$ be a width direction. Then both linear programs
   \[
     \max \{ c x : A x \leq b \} \quad \text{and} \quad \min \{ c x :
     A x \leq b \}
   \]
   are bounded and therefore there are bases $N_1$ and $N_2$ of $A$
   such that $c$ belongs to both cones
   \begin{equation}
     \label{eq:cones-width}
     C_1 \coloneqq \{ y A_{N_1} : y \geq 0 \} \quad \text{and} \quad
     C_2 \coloneqq \{ - y A_{N_2} : y \geq 0 \}
   \end{equation}
   generated by the rows of matrices $A_{N_1}$ and $- A_{N_2}$,
   respectively. But then the linear programs
   \[
     \max \{ c x : A x \leq b' \} \quad \text{and} \quad \min \{ c x :
     A x \leq b' \}
   \]
   are also bounded, whence $w_c(P_{b'})$ is finite.
 \end{proof}

 \noindent The above lemma shows that finite lattice width is a
 property of the matrix $A$. In particular $P_0$ has finite lattice
 width if and only if $P_b$ has finite lattice width for all
 $b$. Conversely, if $P_0$ has infinite lattice width, then $P_b$ also
 has infinite lattice width and therefore contains an integral point
 for all $b$. We can easily recognize whether $P_0$ has infinite
 lattice width. For instance, we can enumerate all possible pairs of
 bases $N_1$ and $N_2$ and check if the cones \eqref{eq:cones-width}
 have a common integral vector. Further we shall not deal with this
 ``trivial'' case and shall consider only those parametric polyhedra,
 for which $w(P_0)$ is finite, and therefore $w(P_b)$ is finite for
 any $b$. We say in this case that the parametric polyhedron $P$ has
 \emph{finite lattice width}.

 \begin{figure}[ht]
   \centering {
     \psset{unit=.6cm}
     \begin{pspicture}(-3,-5)(12.5,3)    
       \pspolygon[linestyle=dashed,linecolor=lightgray,
         fillcolor=vvlg,fillstyle=solid]
         (0,0)(1.5,1.5)(3,2)(8,2)(10,.5)(11,-1)(10,-2.5)(4,-2)(1.5,-1)

       \psline(0,0)(3,3)
       \psline(0,0)(3,-2)

       \pspolygon[linecolor=vlg,fillcolor=vlg,fillstyle=solid]
         (0,0)(-3,3)(-3,-4.5)
       \psline[linecolor=vlg](0,0)(-3,3)
       \psline[linecolor=vlg](0,0)(-3,-4.5)
       \psline[arrows=->](0,0)(-2,0)

       \psline(11,-1)(9,2)
       \psline(11,-1)(9,-4)

       \pspolygon[linecolor=vlg,fillcolor=vlg,fillstyle=solid]
         (11,-1)(8,-3)(8,1)
       \psline[linecolor=vlg](11,-1)(8,-3)
       \psline[linecolor=vlg](11,-1)(8,1)
       \psline[arrows=->](11,-1)(9,-1)

       \rput(-2.5,1.5){$C_1$}
       \rput(-1.5,-.5){$c$}
       \rput(8.5,0){$C_2$}
       \rput(9.5,-1.5){$c$}

       \rput(5,1.5){$P_b$}

       \rput[l](.5,0){$F_{N_1}b$}
       \rput[l](11.5,-1){$F_{N_2}b$}

       \psdots(0,0)(11,-1)
     \end{pspicture}
   }
   \caption{The width direction $c$ and the two cones $C_1$ and
     $C_2$.}
   \label{fig:2}
 \end{figure}
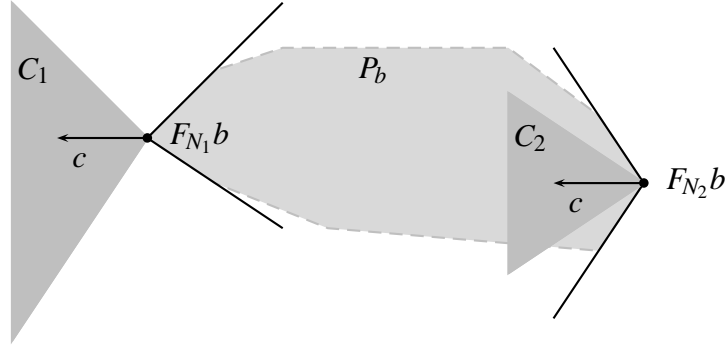

 Now, suppose that $P_b$ is non-empty and let $c$ be its width
 direction. Then there are two bases $N_1$ and $N_2$ such that
 \begin{equation}
   \label{eq:max-min-width}
   \max \{ c x : A x \leq b \} = c F_{N_1} b \quad \text{and} \quad
   \min \{ c x : A x \leq b \} = c F_{N_2} b
 \end{equation}
 and $c$ belongs to the cones $C_1$ and $C_2$ defined by
 \eqref{eq:cones-width}, see Figure~\ref{fig:2}. In fact, equations
 \eqref{eq:max-min-width} hold for any vector $c$ in $C_1 \cap
 C_2$. Thus, the lattice width of $P_b$ is equal to the optimum value
 of the following optimization problem:
 \begin{equation}
   \label{eq:c-ilp}
   \min \{ c (F_{N_1} - F_{N_2}) b : c \in C_1 \cap C_2 \cap \setZ^n
   \setminus \{ 0 \} \}.
 \end{equation}
 The latter can be viewed as an integer programming problem. Indeed,
 the cones $C_1$ and $C_2$ can be represented by some systems of
 linear inequalities, say $c D_1 \leq 0$ and $c D_2 \leq 0$,
 respectively, where $D_1,D_2\in \setZ^{n \times n}$. The minimum
 \eqref{eq:c-ilp} is taken over all integral vectors $c$ satisfying $c
 D_1 \leq 0$ and $c D_2 \leq 0$, except the origin. Since both cones
 $C_1$ and $C_2$ are simplicial, i.e., generated by $n$ linearly
 independent vectors, the origin is a vertex of $C_1 \cap C_2$ and
 therefore can be cut off by a single inequality, for example, $c D_1
 \mathbf{1} \leq -1$, where $\mathbf{1}$ denotes the $n$-dimensional
 all-one vector. It is important that all other integral vectors $c$
 in $C_1 \cap C_2$ satisfy this inequality and therefore remain
 feasible. Thus, the problem \eqref{eq:c-ilp} can be rewritten as
 \[
   \min \{ c (F_{N_1} - F_{N_2}) b : c D_1 \leq 0, \, c D_2 \leq 0, \,
   c D_1 \mathbf{1} \leq -1 \, c \in \setZ^n \}.
 \]
 For a given $b$, this is an integer programming problem. Therefore,
 the optimum value of \eqref{eq:c-ilp} is attained at some vertex of
 the integer hull of the underlying polyhedron
 \begin{equation}
   \label{eq:c-ilp-2}
   \{ c : c D_1 \leq 0, \, c D_2 \leq 0, \, c D_1 \mathbf{1} \leq -1
   \}
 \end{equation}
 \cite{Shevchenko1981} and \cite{HayesL1983} proved that the number of
 vertices of the integer hull of a rational polyhedron is polynomial
 in fixed dimension. Tight bounds for number were presented in
 \citep{CookHKM1992} and \citep{BaranyHoweLovasz92}. This gives rise to
 the next lemma.

 \begin{lemma}
   \label{lem:flat-dirs}
   There is an algorithm that takes as input a rational matrix $A \in
   \setQ^{m \times n}$ of full column rank, which defines a parametric
   polyhedron $P$ of finite lattice width, and computes a set of
   triples $(F_i, G_i, c_i)$ of rational linear transformations $F_i,
   G_i : \setR^m \to \setR^n$ and a non-zero integral row-vector $c_i
   \in \setZ^n$ ($i = 1, \ldots, t$) satisfying the following
   properties. For all $b$, for which $P_b$ is non-empty,
   \begin{enumerate}[(a)]
     \item \label{item:flat-dirs-a}
       $F_i$ and $G_i$ provide, respectively, an upper and lower bound
       on the value of the linear function $c_i x$ in $P_b$, i.e., for
       all $i$,
       \[
         c_i G_i b \leq \min \{ c_i x : x \in P_b \} \leq \max \{ c_i
         x : x \in P_b \} \leq c_i F_i b,
       \]
     \item \label{item:flat-dirs-b}
       the lattice width of $P_b$ is attained along the direction
       $c_i$ for some $i \in \{ 1, \ldots, t \}$ and can be expressed
       as
       \[
         w(P_b) = \min_i \; c_i (F_i - G_i) b.
       \]
     \item \label{item:flat-dirs-c}
       The number $t$ of the triples satisfies the bound
       \begin{equation}
         \label{eq:t-bound}
         t \leq 2 m^{2 n} ( 2 n + 1 )^n ( 24 n^5 \phi )^{n-1},
       \end{equation}
       where $\phi$ is the maximum size of a column in $A$.
   \end{enumerate}
   The algorithm runs in polynomial time if $n$ is fixed.
 \end{lemma}

 \begin{proof}
   In the first step of the algorithm we enumerate all possible bases
   of $A$. Observe that there is at least one basis, since $A$ is of
   full column rank. On the other hand, the total number of possible
   bases is at most $m^n$. Hence, the number of possible pairs of
   bases is bounded by $m^{2 n}$. The algorithm iterates over all
   unordered pairs of bases and for each such pair $\{ N_1, N_2 \}$
   does the following.

   Let $C_1$ and $C_2$ be the corresponding simplicial cones, defined
   by \eqref{eq:cones-width}. These cones can be represented by
   systems of linear inequalities, $c D_1 \leq 0$ and $c D_2 \leq 0$
   respectively, where $D_1, D_2 \in \setZ^{n \times n}$ and the size
   of each inequality is bounded by $4 n^2 \phi$, see \cite[Theorem
   10.2]{Schrijver1986}. As the origin is a vertex of the cone $C_1
   \cap C_2$, it can be cut off by a single inequality; for example,
   $c D_1 \mathbf{1} \leq -1$, where $\mathbf{1}$ stands for the
   $n$-dimensional all-one vector. The size of the latter inequality
   is bounded by $4 n^3 \phi$.

   Thus, there are exactly $2 n + 1$ inequalities in
   \eqref{eq:c-ilp-2} and the size of each inequality is bounded by $4
   n^3 \phi$. This implies that the number of vertices of the integer
   hull of \eqref{eq:c-ilp-2} is at most $2 (2 n + 1)^n (24 n^5
   \phi)^{n-1}$, cf. \citep{CookHKM1992}, and they all can be computed
   in polynomial time if $n$ is fixed, cf. \citep{Hartmann1989}. The
   algorithm then outputs the triple $(F_{N_1}, F_{N_2}, c)$ for each
   vertex $c$ of the integer hull of \eqref{eq:c-ilp-2}, where
   $F_{N_1}$ and $F_{N_2}$ are the linear transformations defined by
   \eqref{eq:fn}. Since there are at most $m^{2 n}$ unordered pairs of
   bases and, for each pair, the algorithm returns at most $2 (2 n +
   1)^n (24 n^5 \phi)^{n-1}$ triples, the total number of triples
   satisfies \eqref{eq:t-bound}, as required. Parts
   \eqref{item:flat-dirs-a} and \eqref{item:flat-dirs-b} of the
   theorem follow directly from our previous explanation.
 \end{proof}

 \noindent The bound \eqref{eq:t-bound} can be rewritten as
 \[
   t = O (m^{2 n} \phi^{n-1})
 \]
 for fixed $n$. Clearly, the greatest common divisor of the components
 of any direction $c_i$ obtained by the algorithm must be equal to
 $1$, as otherwise it would not be a vertex of the integer hull of
 \eqref{eq:c-ilp-2}. This implies, in particular, that the Hermite
 normal form of any of these vectors is just the first unit vector
 $e_1 \in \setR^n$.

 It is also worth mentioning that if $(F_i, G_i, c_i)$ is a triple
 attaining the minimum in Part \eqref{item:flat-dirs-b} of Lemma
 \ref{lem:flat-dirs}, then we have
 \[
   w(P_b) \leq \max \{ c_i x : x \in P_b \} - \min \{ c_i x : x \in
   P_b \} \leq c_i F_i b - c_i G_i b = w(P_b).
 \]
 Hence, Part \eqref{item:flat-dirs-a}, when applied to this triple,
 turns into
 \[
   \min \{ c_i x : x \in P_b \} = c_i G_i b \quad \text{and} \quad
   \max \{ c_i x : x \in P_b \} = c_i F_i b.
 \]

 For our further purposes, it is more suitable to have a \emph{unique}
 width direction for all polyhedra $P_b$ with varying $b$. In fact,
 using Lemma \ref{lem:flat-dirs}, we can partition the set of the
 right-hand sides $b$ into a number of partially open polyhedra such
 that the width direction remains the same for all $b$ belonging to
 the same region of the partition.

 \begin{theorem}
   \label{thm:flat-dirs}
   Let $P$ be a parametric polyhedron of finite lattice width, defined
   by a matrix $A \in \setQ^{m \times n}$ of full column rank. Let $Q
   \subseteq \setR^m$ be a rational partially open polyhedron such
   that $P_b$ is non-empty for all $b \in Q$. We can compute---in
   polynomial time, if $n$ is fixed---a partition of $Q$ into a number
   of partially open polyhedra $Q_1, \ldots, Q_t$ and, for each $i$,
   find a triple $(F_i, G_i, c_i)$ of linear transformations $F_i, G_i
   : \setR^m \to \setR^n$ and a non-zero vector $c_i \in \setZ^n$,
   such that
   \[
     \min \{ c_i x : x \in P_b \} = c_i G_i b, \quad \max \{ c_i x : x
     \in P_b \} = c_i F_i b,
   \]
   and
   \[
     w(P_b) = w_{c_i}(P_b) = c_i (F_i - G_i) b \quad \text{for all } b
     \in Q_i.
   \]
   If $\phi$ denotes the maximum size of a column in $A$, then $t = O
   ( m^{2 n} \phi^{n-1} )$.
 \end{theorem}
 
 \begin{remark}   
   The statement of Theorem \ref{thm:flat-dirs} is very analogous to
   \cite[Lemma~3.1]{Kannan1992}. However, there are several crucial
   differences. First, the number of regions in the partition obtained
   by Kannan's algorithm is exponential in $n$ and the affine
   dimension $j_0$ of the polyhedron $Q$. Our algorithm yields a
   partition that is exponential in $n$ only, hence polynomial if $n$
   is fixed. Also our algorithm runs in polynomial time if $n$ is
   fixed but $j_0$ may vary. At last, the width directions $c_i$
   obtained by the Kannan's algorithm satisfy, for any $b \in Q_i$,
   \[
     \text{either} \quad w_{c_i}(P_b) \leq 1 \quad \text{or} \quad
     w_{c_i}(P_b) \leq 2 \, w(P_b)
   \]
   In contrast, we compute the \emph{exact} width direction for each
   region in the partition. While the exact computation of width
   directions does not help much from the algorithmic point of view,
   removing the restriction on the dimension of $Q$ turns out to be
   the main step towards the claimed generalization of Kannan's
   algorithm.
 \end{remark}

 \begin{proof}[Proof of Theorem \ref{thm:flat-dirs}]
   First, we exploit the algorithm of Lemma \ref{lem:flat-dirs} to
   obtain the triples $(F_i, G_i, c_i)$, $i = 1, \ldots, t$, with $t =
   O (m^n \phi^{n-1})$. These triples provide the possible width
   directions of the parametric polyhedron $P$. For each $i = 1,
   \ldots, t$, we define a partially open polyhedron $Q_i$ by the
   inequalities
   \begin{align*}
     & c_i (F_i - G_i) b < c_j (F_j - G_j) b, \quad j = 1, \ldots, i -
     1, \\ & c_i (F_i - G_i) b \leq c_j (F_j - G_j) b, \quad j = i +
     1, \ldots, t.
   \end{align*}
   Thus,
   \[
     \min_j \, c_j (F_j - G_j) b = c_i (F_i - G_i) b
   \]
   for all $b \in Q_i$. We claim that the intersections of the
   partially open polyhedra $Q_i$ with $Q$ give the required
   partition.
   
   Indeed, let $b \in Q$ and let $\mu$ be the minimum value of $c_i
   (F_i - G_i) b$, $i = 1, \ldots, t$. Let $I$ denote the set of
   indices $i$ with $c_i (F_i - G_i) b = \mu$. Then $b \in Q_{i_0}$,
   where $i_0$ is the smallest index in $I$. Yet, suppose that $b \in
   Q$ belongs to two partially open polyhedra, say $Q_i$ and
   $Q_j$. Without loss of generality, we can assume $i < j$. But then
   we have
   \[
     c_i (F_i - G_i) b \leq c_j (F_j - G_j) b < c_i (F_i - G_i) b,
   \]
   where the first inequality is due to the fact $b \in Q_i$ and the
   second inequality follows from $b \in Q_j$; both together are a
   contradiction.

   For the width directions, Lemma \ref{lem:flat-dirs} implies that
   \[
     w(P_b) = \min_j \, c_j (F_j - G_j) b = c_i (F_i - G_i) b
   \]
   for all $b \in Q_i \cap Q$. This completes the proof.
 \end{proof}

 \noindent Theorem \ref{thm:flat-dirs} provides a unique width
 direction for each region $Q_i$ of the partition. This resolves the
 first problem of adapting the algorithm for integer linear
 programming in fixed dimension to the case of varying $b$, which was
 addressed in the introduction. However, we still need to deal with
 the hyper-planes $\{ x : c_i x = \lceil \beta \rceil + j \}$, where
 $\beta$ is the optimum value of the linear program $\min \{ c_i x : x
 \in P_b \}$. As mentioned above, $\beta$ can be expressed as a linear
 transformation of $b$, namely $\beta = c_i G_i b$. But $\lceil \beta
 \rceil$ is no more a \emph{linear} function of $b$, which makes the
 recursion complicated. \cite{Kannan1992} showed how to tackle this
 problem. We discuss it in the next section.

\section{Partitioning theorem and parametric integer programming}
\label{sec:partit}

 The proof of the following structural result follows from the proof
 of \cite[Theorem 4.1]{Kannan1992} if it is combined with Theorem
 \ref{thm:flat-dirs}.

 \begin{theorem}
   \label{thm:structural}
   Let $P$ be a parametric polyhedron of finite lattice width, defined
   by a rational matrix $A \in \setQ^{m \times n}$ of full column
   rank. Let $Q \subseteq \setR^m$ be a rational partially open
   polyhedron such that $P_b$ is non-empty for all $b \in Q$. We can
   compute---in polynomial time, if $n$ is fixed---a partition of $Q$
   into sets $S_1, \ldots, S_t$, and for each $i$, find a number of
   unimodular transformations $U_{ij} : \setR^n \to \setR^n$ and
   affine transformations $T_{ij} : \setR^m \to \setR^n$, $j = 1,
   \ldots, k_i$ such that
   \begin{enumerate}[(a)]
     \item
       each $S_i$ is the integer projection of a partially open
       polyhedron, $S_i = S'_i / \setZ^{l_i}$;
     \item
       for any $b \in S_i$, $P_b \cap \setZ^n \neq \emptyset$ if and
       only if $P_b$ contains $U_{ij} \lceil T_{ij} b \rceil$ for some
       index $j$;
     \item
       the following bounds hold:
       \[
         t = O ((m^{2 n} \phi^{n-1})^{n \overline{\omega}(n)}), \quad
         l_i = O (\overline{\omega}(n)), \quad k_i = O (2^{n^2 / 2}
         \overline{\omega}(n)), \quad i = 1, \ldots, t,
       \]
       where $\phi$ denotes the maximum size of a column in $A$ and
       $\overline{\omega}(n) = \prod_{i=1}^n \omega(n)$.
    \end{enumerate}
 \end{theorem}

 We do not repeat Kannan's proof here but give an intuition of why it
 is true in dimension 2. By Theorem~\ref{thm:flat-dirs} we can assume
 that the width-direction is invariant for all $b \in Q$ and by
 applying a unimodular transformation we can further assume that this
 width-direction is the first unit-vector $e_1$, see Figure
 \ref{fig:dim2}.
 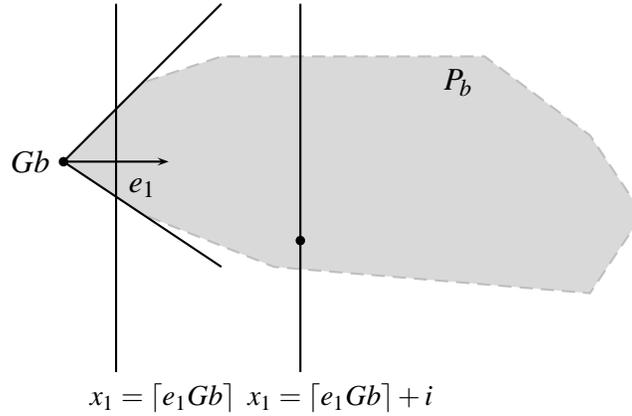
\begin{figure}[ht]
   \centering {
     \psset{unit=.7cm}
     \begin{pspicture}(-1,-5)(11,3)    
       \pspolygon[linestyle=dashed,linecolor=lightgray,
         fillstyle=solid,fillcolor=vvlg]
         (0,0)(1.5,1.5)(3,2)(8,2)(10,.5)(11,-1)(10,-2.5)(4,-2)(1.5,-1)
       \psline(0,0)(3,3)
       \psline(0,0)(3,-2)

       \psline[arrows=->](0,0)(2,0)

       \psline(1,-4)(1,3)
       \psline(4.5,-4)(4.5,3)
       \psdot*(4.5,-1.5)

       \rput(1.5,-.5){$e_1$}
       \rput(7.5,1.5){$P_b$}

       \rput[l](-1,0){$G b$}
       { \small
         \rput[l](.5,-4.5){$x_1=\lceil e_1 G b \rceil$}
         \rput[l](3.5,-4.5){$x_1=\lceil e_1 G b \rceil+i$}
       }
    
       \psdots(0,0)
     \end{pspicture}
     \caption{Illustration of Theorem~\ref{thm:structural} in dimension 2.}
     \label{fig:dim2}
   }
 \end{figure}

 Lemma~\ref{lem:flatness-app} tells that $P_b$ contains an integral
 point if and only if there exists an integral point on the lines $x_1
 = \lceil e_1 G b \rceil + j$ for $j = 0, \ldots, \omega(2)$. The
 intersections of these lines with the polyhedron $P_b$ are
 1-dimensional polyhedra. Some of the constraints $a x \leq \beta$ of
 $A x \leq b$ are ``pointing upwards'', i.e., $a e_2 < 0$, where $e_2$
 is the second unit-vector. Let $a_1 x_1+a_2x_2 \leq \beta$ be a
 constraint pointing upwards such that the intersection point $(\lceil
 e_1 G b \rceil + j, y)$ of $a_1 x_1+a_2x_2 = \beta$ with the line
 $x_1 = \lceil e_1 G b \rceil + j$ has the largest second
 component. The line $x_1 = \lceil e_1 G b \rceil + j$ contains an
 integral point in $P_b$ if and only if $(\lceil e_1 G b \rceil + j,
 \lceil y \rceil)$ is contained in $P_b$. This point is illustrated in
 Figure~\ref{fig:dim2}. By choosing the highest constraint pointing
 upwards for each line $x_1 =\lceil e_1 G b \rceil + j$, we partition
 the set of right-hand sides into polynomially many integer
 projections of partially open polyhedra.

 In order to express the candidate solution $(\lceil e_1 G b \rceil +
 j, \lceil y \rceil)$ in the form described in the theorem, observe
 that
 \[
   y =  (\beta - a_1 x_1) / a_2. 
 \]
 Since $x_1$ is an integer and
 \[
   x_1 = \lceil e_1 G b \rceil + j = e_1 G b + j + \gamma
 \]
 for some $\gamma \in [0,1)$, we can rewrite the equation $\lceil y
 \rceil = \lceil (\beta - a_1 x_1) / a_2 \rceil$ as
 \begin{align*}
   \lfloor a_1 / a_2 \rfloor x_1 + \lceil y \rceil &= \lceil \beta /
   a_2 - \{ a_1 / a_2 \} x_1 \rceil \\ &= \lceil \beta / a_2 - \{ a_1
   / a_2 \} (e_1 G b + j) - \{ a_1 / a_2 \} \gamma \rceil,
 \end{align*}
 where $\{ a_1 / a_2 \}$ denotes the fractional part of $a_1 /
 a_2$. Since $\{ a_1 / a_2 \} \gamma$ lies between 0 and 1, it
 suffices to check independently two different possibilities, namely,
 \[
   \lfloor a_1 / a_2 \rfloor x_1 + \lceil y \rceil = \lceil \beta /
   a_2 - \{ a_1 / a_2 \} (e_1 G b + j) \rceil
 \]
 and
 \[
   \lfloor a_1 / a_2 \rfloor x_1 + \lceil y \rceil = \lceil \beta /
   a_2 - \{ a_1 / a_2 \} (e_1 G b + j) - 1 \rceil.
 \]
 Combined with $x_1 = \lceil e_1 G b \rceil + j$, each of the above
 equations yields a unimodular system with respect to the variables
 $x_1$ and $\lceil y \rceil$, with the right-hand side being the
 round-up of an affine transformation of $b$. We refer the reader to
 \citep{Kannan1992} to see the complete proof for arbitrary dimension.

\subsection*{\pmb{$\forall \,\exists$}-statements} 

 Theorem \ref{thm:structural} gives rise to a polynomial algorithm for
 testing sentences of the form
 \begin{equation}
   \label{eq:pilp-appl}
   \forall b \in Q / \setZ^p \quad \exists x \in \setZ^n : \quad A x
   \leq b,
 \end{equation}
 when $p$ and $n$ are fixed. This algorithm was first described by
 \cite{Kannan1992} but he required, in addition, the affine dimension
 of $Q$ to be fixed. Our improvement follows basically from the
 improvement in the partitioning theorem, while the algorithm itself
 remains exactly the same. We describe it here for the sake of
 completeness.  First observe that we can assume that $A$ has full
 column rank. Otherwise we can apply a unimodular transformation of
 $A$ from the right to obtain a matrix $[A' \mid 0]$, where $A'$ has
 full column rank.

 The idea is as follows: First we run the algorithm of Theorem
 \ref{thm:structural} on input $A$ and $Q' \subseteq \setR^m$, where
 $Q'$ is the set of vectors $b$, for which the system $A x \leq b$ has
 a solution. Then we consider each set $S_i$ returned by the algorithm
 of Theorem \ref{thm:structural} independently. For each $b \in S_i$
 we have a fixed number of candidate solutions for the system $A x
 \leq b$, defined via unimodular and affine transformations as $U_{ij}
 \lceil T_{ij} b \rceil$. Each rounding operation can be expressed by
 introducing an integral vector: $z = \lceil T_{ij} b \rceil$ is
 equivalent to $T_{ij} b \leq z < T_{ij} b + \mathbf{1}$. We need only
 a constant number of integer variables to express all candidate
 solutions plus a fixed number of integer variables to represent the
 integer projections $S_i = S'_i / \setZ^{l_i}$. It remains to solve a
 number of mixed-integer programs, to which we also include the
 constraints $(b, y) \in Q$, $y \in \setZ^p$.

 \begin{theorem}
   \label{thm:main}
   There is an algorithm that, given a rational matrix $A \in \setQ^{m
   \times n}$ and a rational polyhedron $Q \subseteq \setR^{m+p}$,
   decides the sentence \eqref{eq:pilp-appl}. The algorithm runs in
   polynomial time if $p$ and $n$ are fixed.
 \end{theorem}

 \begin{proof}
   Let $P$ be a parametric polyhedron defined by the matrix
   $A$. First, we exploit the Fourier--Motzkin elimination procedure
   to construct the polyhedron $Q' \subseteq \setR^m$ of the
   right-hand sides $b$, for which the system $A x \leq b$ has a
   (fractional) solution. For each inequality $a b \leq \beta$,
   defining the polyhedron $Q'$, we can solve the following
   mixed-integer program
   \[
     a b > \beta, \quad (b, y) \in Q, \quad y \in \setZ^p,
   \]
   and if any of these problems has a feasible solution $(y, b)$, then
   $b$ is a vector in $Q / \setZ^p$, for which the system $A x \leq b$
   has no integral solution. Hence, we can terminate and output ``no''
   (with $b$ being a certificate).

   We can assume now that for all $b \in Q / \setZ^p$ the system $A x
   \leq b$ has a fractional solution. By applying the algorithm of
   Theorem \ref{thm:structural}, we construct a partition of $Q'$ into
   the sets $S_1, \ldots, S_t$, where each $S_i$ is the integer
   projection of a partially open polyhedron, $S_i = S'_i /
   \setZ^{l_i}$. Since $n$ is fixed, the $l_i$ are bounded by some
   constant, $i = 1, \ldots, t$. Furthermore, for each $i$, the
   algorithm constructs unimodular transformations $U_{ij}$ and affine
   transformations $T_{ij}$, $j = 1, \ldots, k_i$, such that $P_b$,
   with $b \in S_i$, contains an integral point if and only if $U_{ij}
   \lceil T_{ij} b \rceil \in P_b$ for some $j$. Again, $k_i$ is fixed
   for a fixed $n$, $i = 1, \ldots, t$.

   The algorithm will consider each index $i$ independently. For a
   given $i$, $S_i$ can be described as the set of vectors $b$ such
   that
   \[
     (b, z) \in S'_i
   \]
   has a solution for some integer $z \in \setZ^{l_i}$. This can be
   expressed in terms of linear constraints, as $S'_i$ is a partially
   open polyhedron. Let $x_j = U_{ij} \lceil T_{ij} b \rceil$. The
   points $x_j$ can be described by linear inequalities as
   \[
     T_{ij} b \leq z_j < T_{ij} b + \mathbf{1}, \quad x_j = U_{ij}
     z_j,
   \]
   where $\mathbf{1}$ is the all-one vector. Then $P_b$ does not
   contain an integral point if and only if $x_j \notin P_b$ for all
   $j = 1, \ldots, k_i$. In this case, each $x_j$ violates at least
   one constraint in the system $A x \leq b$. We consider all possible
   tuples $I$ of $k_i$ constraints from $A x \leq b$. Obviously, there
   are only $m^{k_i}$ such tuples, that is, polynomially many in the
   input size. For each such tuple, we solve the mixed-integer program
   \[ \begin{array}{ll}
     (b, y) \in Q, \, (b, z) \in S'_i, \medskip \\ T_{ij} b \leq z_j <
     T_{ij} b + \mathbf{1}, & j = 1, \ldots, k_i, \medskip \\ x_j =
     U_{ij} z_j, & j = 1, \ldots, k_i, \medskip \\ a_{i_j} x_j >
     b_{i_j}, & j = 1, \ldots, k_i, \medskip \\ y \in \setZ^p, \, z
     \in \setZ^{l_i}, \, z_j \in \setZ^n, & j = 1, \ldots, k_i,
   \end{array} \]
   where $a_{i_j} x \leq b_{i_j}$ is the $j$-th constraint in the
   chosen tuple. Each such mixed-integer program can be solved in
   polynomial time since the number of integer variables is fixed (in
   fact, there are at most $(k_i + 1) n + l_i$ integer variables).

   If there is a feasible solution $b$ to one of these mixed-integer
   programs, then the answer to the original problem is ``no'' (with
   $b$ being a certificate). If all these mixed-integer programs are
   infeasible, the answer is ``yes''.
 \end{proof}

 \begin{remark}
   \label{rem:1}
   We would like to point out that Theorem \ref{thm:main} can also be
   proved differently. \cite{Bell77} and \cite{Scarf77} showed that if
   a system of linear inequalities $A x \leq b$ has no integral
   solution, then there is already a subsystem of at most $2^n$
   inequalities that is infeasible in integer variables; see also
   \cite[Theorem 16.5]{Schrijver1986}. Applied to \eqref{eq:pilp} this
   means that there exists a $b \in Q / \setZ^p$ such that the system
   $A x \leq b$ has no integral solution, if and only there exist a $b
   \in Q / \setZ^p$ and a subsystem $A' x \leq b'$ of $A x \leq b$
   with at most $2^n$ inequalities, which is infeasible in integer
   variables. Here $b'$ is the projection of $b$ into the according
   space. Since $n$ is a constant, we can try out all $\binom{m}{2^n}$
   different subsystems of $A x \leq b$ and apply to each of these
   parametric polyhedra Kannan's algorithm to decide $\forall\,
   \exists$-statements.

   However, a similar argument does not yield our extension (Theorem
   \ref{thm:structural}) of Kannan's partitioning theorem itself,
   which associates to each $b$ a fixed set of candidate integer
   solutions, depending on the partially open polyhedron of the
   partitioning, in which $b$ is contained.
 \end{remark}

\section{Integer programming gaps} 
 
 Now, we describe how Theorem \ref{thm:main} can be applied to compute
 the maximum integer programming gap for a family of integer programs.
 Let $A \in \setQ^{m \times n}$ be a rational matrix and let $c \in
 \setQ^n$ be a rational vector. Let us consider the integer programs
 of the form
 \begin{equation}
   \label{eq:ilp}
   \max \{ c x : A x \leq b, \, x \in \setZ^n \},
 \end{equation}
 where $b$ is varying over $\setR^m$. The corresponding linear
 programming relaxations are then
 \begin{equation}
   \label{eq:lp}
   \max \{ c x : A x \leq b \}.
 \end{equation}
 Consider the following system of inequalities:
 \begin{align*}
   c x & \geq \beta, \medskip \\ A x & \leq b.
 \end{align*}
 Given a vector $b$ and a number $\beta$, there exists a feasible
 fractional solution of the above system if and only if the linear
 program \eqref{eq:lp} is feasible and its value is at least $\beta$.
 The set of pairs $(\beta, b) \in \setR^{m+1}$, for which the above
 system has a fractional solution, is a polyhedron in $\setR^m$ and
 can be computed by means of Fourier--Motzkin elimination, in
 polynomial time if $n$ is fixed. Let $Q$ denote this polyhedron.

 Suppose that we suspect the maximum integer programming gap to be
 smaller than $\gamma$. This means that, whenever $\beta$ is an
 optimum value of \eqref{eq:lp}, the integer program \eqref{eq:ilp}
 must have a solution of value at least $\beta -
 \gamma$. Equivalently, the system
 \begin{align}
   \label{eq:gap-system}
   c x & \geq \beta - \gamma, \medskip \\ A x &\leq b, \nonumber
 \end{align}
 must have an integral solution. If there exists $(b, \beta) \in Q$
 such that \eqref{eq:gap-system} has no integral solution, the integer
 programming gap is bigger than $\gamma$. We also need to ensure that
 for a given $b$, the integer program is feasible, i.e., the system $A
 x \leq b$ has a solution in integer variables.

 Now, this is exactly the question for the algorithm of Theorem
 \ref{thm:main}: Is there a $(\beta, b) \in Q'$ such that the system
 \eqref{eq:gap-system} has no integral solution, but there exists $y
 \in \setZ^n$ such that $A y \leq b$? Here $Q' = Q - \gamma (1, 0)$ is
 the appropriate translate of the set $Q$. If the algorithm answers
 ``no'' with the certificate $b$, then the integer program
 \eqref{eq:ilp}, with the right-hand side $b$, has no solution of
 value greater than $\beta - \gamma$ while being feasible. On the
 other hand, $(\beta, b) \in Q$, thus the corresponding linear
 solution has optimum value of at least $\beta$. We can conclude that
 the maximum integer programming gap is greater than $\gamma$. This
 gives us the following theorem.

 \begin{theorem}
   \label{thm:gap}
   There is an algorithm that, given a rational matrix $A \in \setR^{m
   \times n}$, a rational row-vector $c \in \setQ^n$ and a number
   $\gamma$, checks whether the maximum integer programming gap for
   the integer programs \eqref{eq:ilp} defined by $A$ and $c$ is
   bigger than $\gamma$. The algorithm runs in polynomial time if the
   rank of $A$ is fixed.
 \end{theorem}

 \noindent Using binary search, we can also find the \emph{minimum}
 possible value of $\gamma$, hence the maximum integer programming
 gap.

\end{document}